\documentclass[reqno]{amsart}

\usepackage{mathbbol,mathrsfs}
\usepackage{amsfonts,amssymb,amsmath,amsthm,amscd}
\usepackage{latexsym}
\usepackage[mathscr]{eucal}
\usepackage{wrapfig,graphics,graphicx,epic,epsfig}
\usepackage{ifthen}
\usepackage{oldgerm,units}
\usepackage{pstricks,pst-node}

 \input xy
\xyoption{all}

 \makeatletter
\renewcommand*\env@matrix[1][*\c@MaxMatrixCols c]{%
  \hskip -\arraycolsep
  \let\@ifnextchar\new@ifnextchar
  \array{#1}}
\makeatother
\def\dispace{\setlength{\itemsep}{2pt}}

\newcommand{\joverline}[2]{%
  \mathord{
    \vbox{\offinterlineskip
      \halign{##\cr
        $\scriptscriptstyle#1$\hrulefill\cr
        \noalign{\kern.4ex}
        $\; #2 \,$\cr
      }%
    }%
  }%
}

\newtheorem{theorem}{Theorem}[section]

\newtheorem{df}[theorem]{Definition}

\newtheorem{conjecture}[theorem]{Conjecture}

\newtheorem{example}[theorem]{Example}
\newtheorem{remark}[theorem]{Remark}

\def\al{\alpha}
\def\bt{\beta}

\newcommand\INV{{\operatorname{INV}}}

\def\snor{strictly normal}

\def\qI{\mathcal{I}}
\def\qIr{\mathcal{I}^r}
\def\qIl{\mathcal{I}^\ell}

\newcommand{\ds}[1]{\, {#1} \, }
\newcommand{\dss}[1]{\quad {#1} \quad }

\newcommand{\mcong}[2]{\,^{#2}\hskip -.25ex{#1}}

\def\Mat{\operatorname{Mat}}
\def\MatnF{\Mat_n(F)}
\def\INV{\Mat_n(F)^{\times}}
\def\QI{\operatorname{QI}}

\def\GLn{\operatorname{GL}_n}

\def\SL{\operatorname{SLS}}
\def\SLn{\operatorname{SL}_n}
\def\SLnI{\operatorname{SLS}_{n}^{\one}}

\def\QSL{\SLS}
\def\BQSLn{\overline{\SL}_n}
\def\BQSLtwo{\overline{\SL}_2}
\def\BQSL{\overline{\SL}}
\def\SLmonr{\operatorname{\BQSL}^{\; r}_{\qI;n}}
\def\SLmonl{\operatorname{\BQSL}^{\; \ell}_{\qI;n}}
\def\SLmonlA{\operatorname{\BQSL}^{\; \ell}_{A;n}}
\def\SLmonrA{\operatorname{\BQSL}^{\; r}_{A;n}}

\def\SLmonA{\operatorname{\BQSL}_{A;n}}

\def\QSLn{{\QSL}_n}

\def\tJn{\operatorname{SN}_{n}}

\newcommand\GLnA[1]{{\operatorname{GL}}{[1]}}
\newcommand\SLA[1]{{\operatorname{SL}}{[1]}}

\def\semirings0{semirings$^\dagger$}
\def\domain0{domain$^\dagger$}

\def\domains0{domains$^\dagger$}
\def\field0{semifield$^\dagger$}
\def\fields0{semifields$^\dagger$}

\def\mfS{\operatorname{S}}

\def\pSkip{\vskip 1.5mm \noindent}
\def\sSkip{\vskip 3mm \noindent}

\def\pipeGS{{\underset{\operatorname{\, gs }}{\mid}}}
\def\lmod{\mathrel   \pipeGS \joinrel\joinrel \joinrel =}
\def\lmodg{\lmod}

\def\sgn{  \operatorname{sgn}}

\def\nucong{\cong_\nu}

\def\nuge{\ge_\nu}
\def\nule{\le_\nu}
\def\nul{<_\nu}
\def\nug{>_\nu}

\newcommand{\etype}[1]{\renewcommand{\labelenumi}{(#1{enumi})}}
\def\eroman{\etype{\roman} \dispace}

\def\tGz{\mathcal G_\zero}

\def\({\left(}
\def\){\right)}

\def\htR{\widehat {R}}
\def\htRR{\widehat {\R}}

\def\brA{\overline{A}}
\def\brB{\overline{B}}

\def\trn{{\operatorname{t}}}

\def\tGz{\mathcal G_\zero}

\newcommand{\trop}[1]{\mathcal{#1}}

\newcommand{\tG}{\trop{G}}

\newcommand{\tS}{\trop{S}}
\newcommand{\tT}{\trop{T}}

\def\per{\operatorname{det}}
\def\bid{\operatorname{bidet}}

\def\a{\alpha}

\def\sig{\sigma}

\def\SLS{\operatorname{SLS}}
\def\one{\mathbb{1}}
\def\zero{\mathbb {0}}

\def\um{I}

\def\nb{\nabla}
\def\nbnb{{\nabla\nabla}}

\def\id{\operatorname{id}}

\newcommand{\adj}[1]{\operatorname{adj}(#1)}

\def\Q{\mathbb Q}

\def\R{\mathbb R}

\newtheorem{thm}[theorem]{Theorem}
\newtheorem*{thm*}{Theorem}
\newtheorem*{dig*}{Digression}

\newtheorem{cor}[theorem]{Corollary}

\newtheorem{lem}[theorem]{Lemma}
\newtheorem{rem}[theorem]{Remark}
\newtheorem{prop*}{Proposition}

\newtheorem{prop}[theorem]{Proposition}
\newtheorem{defn}[theorem]{Definition}
\newtheorem*{examp*}{Example}
\newtheorem*{examples*}{Examples}
\newtheorem*{remark*}{Remark}

\newtheorem*{defn*}{Definition}
\newtheorem*{note*}{Note}

\DeclareMathAlphabet{\mathbbold}{U}{bbold}{m}{n}

\newcommand{\TR}{\mathbbold{R}}

\begin{document}
\title[Supertropical $\SLn$]{Supertropical  $\SLn$}

\author[Z.~Izhakian]{Zur Izhakian}
\address{Institute  of Mathematics,
 University of Aberdeen, AB24 3UE,
Aberdeen,  UK. }
    \email{zzur@abdn.ac.uk}

 \author[ A.~Niv]{Adi Niv}
\address{Mathematics Unit, SCE College of Engineering, 56 Bialik St., Beer-Sheva 84100 Israel}
\email{adini1@sce.ac.il}

\author[L.~Rowen]{Louis Rowen}
\address{Department of Mathematics, Bar-Ilan University, Ramat-Gan 52900,
Israel.} \email{rowen@math.biu.ac.il}

\thanks{The   authors thank Oliver Lorscheid for raising the question concerning the tropical SL monoid.  (Yale, May  2014).}
\thanks{The   authors also thank the reviewer for constructive suggestions on the original submission.}

\thanks{The research of the first author  has been sported by the Research Councils UK (EPSRC), grant no EP/N02995X/1.}

\thanks{The second author  has been sported by the French Chateaubriand grant and INRIA postdoctoral fellowship.}

\thanks{The third author would like to thank the University of
Virginia for its support during the preparation of this work}



\date{\today}




\begin{abstract} Extending earlier work on supertropical adjoints and applying symmetrization,
we provide a symmetric supertropical version $\QSLn$ of the special
linear group $\SLn$, which we partially decompose into submonoids,
based on ``quasi-identity'' matrices, and we display maximal
sub-semigroups of $\QSLn$. We also study the monoid generated by
$\QSLn$ and its natural submonoids. Several illustrative examples
are given of unexpected behavior. We describe the action of
elementary matrices
 on $\QSLn$, which enables one to connect different matrices in $\QSLn$,
 but in a weaker sense than the classical situation. 
\end{abstract}

\maketitle


\numberwithin{equation}{section}

\section*{Introduction} This paper rounds out \cite{IzhakianRowen2008Matrices,IzhakianRowen2009Equations}, its main objective being to lay out the
foundations of the theory of~$\SLn$ in tropical linear algebra.
Given any semiring~$R$, one can define the
 matrix semiring, comprised of matrices~$A = (a_{i,j}) $  with entries in~$R,$ where the addition and
multiplication  of matrices are induced from~$R$ as in the familiar
ring-theoretic matrix construction.

The classical definition of~$\GLn$ is the set of invertible
matrices, which coincides  with the set of nonsingular matrices. Then,
the set~$\SLn\subseteq \GLn$ is the set of matrices with
determinant~$1$, in which case~$A^{-1} = \adj A$. In particular,
this is the group generated by the elementary matrices~$E_{i,j}$,
which differ from the identity matrix by one nondiagonal nonzero
entry in the~$(i,j)$ position. These elementary matrices play a
fundamental role in linear algebra and K-theory. Our basic goal is
to find the tropical analog, containing the elementary matrices and
preferably all matrices of determinant~$1$, which raises various
difficulties. Tropical algebra is based on the max-plus algebra, for
which negation does not exist and its underlying semiring structure
is idempotent. For purposes of motivation, we consider matrices over
an ordered semifield~$F$ (i.e.,~$F \setminus\{\zero\}$ is a
multiplicative group), such as the max-plus algebra~$\Q_{\max}$
or~$\R_{\max}$ (noting that in this case the multiplicative
identity~$\one$ is~$0$, and the additive identity~$\zero$
is~$-\infty$). Later on we switch to the supertropical language,
which is more convenient.

Invertibility of matrices (in its classical sense) is quite
restricted in the (super)tropical setting. In view of Remark~2.1
below, the matrices~$E_{i,j}$ are not invertible, and therefore do
not generate any permutation matrices. Nevertheless, the
matrices~$E_{i,j}$ are tropically nonsingular (to be defined
presently) of determinant~$\one$. Applying a permutation matrix to a
set of vectors in~$F^{(n)}$ merely rearranges the coordinates,
whereas applying a diagonal matrix rotates the rays, or thought
another way, rescales the coordinates. From this point of
view,~$E_{i,j}$ has considerable geometric significance, and should
be in any serious tropical version of~$\SL_n $.

The classical determinant of~$A=(a_{i,j})$ is no longer
available in the tropical setting, due to its lack of negations. One of the challenges of tropical
matrix theory has been to introduce a viable analog of the determinant, given these limitations.
In~\cite{Pl}, the determinant was defined as usual, using tropical
operations and permutation signs. In~\cite{IzhakianRowen2007SuperTropical}, the permanent (called tropical determinant) was used as a
substitute,
  given as
 \begin{equation}\label{detdf}\per(A) = \sum_{\pi \in \mfS_n}    \prod_{i=1}^n a_{i,\pi(i)},\end{equation} and
  formulated in
\cite{B} as the optimal assignment problem. This
  approach has roots  going back to~\cite{RS,St}, and~\cite{Pl} also
  studied the optimal assignment problem by means of the permanent.

Using the  permanent leads to a corresponding definition of the
adjoint matrix and the matrix (cf.~\cite{zur05TropicalAlgebra}) $$A^\nb : = {\per(A)}^{-1} \adj{A},$$
 and was used in~\cite{IzhakianRowen2008Matrices,IzhakianRowen2009Equations} to build
a theory parallel to the classical theory. In particular, a
matrix~$A$ is \textbf{nonsingular} if $ \per(A)$ is ``tangible,''
and these matrices are exactly those of full row rank, by
\cite[Corollary~6.6]{IzhakianRowen2008Matrices}. So one is led to
define~$\SL_n$ to be the set of nonsingular matrices with
determinant~$\one$, in which case~$A^\nabla = \adj A$.

Although the supertropical language is not strictly needed for our definition of~$\SL_n$, it makes the
statements easier, and ``supertropical matrix theory" has led to results in linear algebra unavailable in
other tropical versions, such as equality of matrix ranks, a natural analog of the Cayley-Hamilton theorem,
solutions of eigenvalues, etc., as indicated in~\cite{IzhakianRowen2009TropicalRank, IzhakianRowen2008Matrices, IzhakianRowen2009Equations}.

\begin{df}\label{norm} A matrix~$A=(a_{i,j})$ is \textbf{definite} if the identity permutation is the unique dominant permutation in~\eqref{detdf},
with~$a_{i,i} = \one$ for all $i$. If~$A$ is definite and~$a_{i,j}\leq \one$ for all $i\ne j$, then~$A$ is \textbf{normal}.~$A$ is \textbf{strictly normal} when all these inequalities are strict.\end{df}

All matrices in~$E_{i,j}$ are definite for all~$i,j$, and every
definite matrix is in~$\SL_n$. The set~$\SL_n$ also contains all
permutation matrices and all diagonal matrices of
determinant~$\one$. It has long been known (see~\cite{St} for
instance), that when~$A_1A_2$ is nonsingular, then it has a unique
dominating permutation and~$\det(A_1A_2) = \det(A_1) \det(A_2)$.
Thus, a nonsingular  product of two matrices in~$\SL_n$   is
in~$\SL_n$, and a nonsingular  product of  two definite matrices is
definite.

Unfortunately,~$\SL_n$  is no longer a group (or even a monoid),
since it need not be closed under tropical matrix multiplication;
for example,  for non-definite matrices,

\begin{equation}\label{badprod} \left(\begin{array}{cc}\one&a\\\zero&\one\end{array}\right)
\left(\begin{array}{cc}\one&\zero\\b&\one\end{array}\right)=\left(\begin{array}{cc}\one+ab&a\\b&\one\end{array}\right).\end{equation} Nevertheless, tropical matrix multiplication
is closed if the multiplicands are strictly normal.

Thus, one of our main objectives is to study $\SL_n$ via related
monoids. We need   a suitable monoid to work with,
cf.~Definition~\ref{symmon}, in order to have a proper
 algebraic structure to progress with K-theory.

One can expand $\SL_n$ a bit by means of an approach of Akian,
 Gaubert, and  Guterman \cite{AGG}, and the Max-plus group \cite{Pl}. They had
 already refined the determinant
  by
 distinguishing between the even and odd permutations in defining the
 \textbf{bideterminant}; also see~\cite{BCOQ}.    A
related approach is given in~\cite{AGG1}. In this context, a matrix
is \textbf{(symmetrically) singular} if
$$\sum_{\text{odd }\pi\in S_n}  \prod_{i=1}^n a_{i,\pi(i) }=
 \sum_{\text{even }\pi\in S_n}  \prod_{i=1}^n  a_{i,\pi(i) }.$$ This yields
  a symmetrized version of~$\SL_n$ in Definition~\ref{genmon}, permitting
symmetrically nonsingular matrices.

This also leads to a subtle distinction, since a singular
 matrix in the supertropical sense (which is a tropicalization of
 a singular matrix over a Puiseux series) could be nonsingular in
 the symmetrized sense.

Consider for example the singular Puisseux matrix $ A=
\(\begin{smallmatrix}
                 t &  t & 0\\
                0 & t& t\\
(2-i)t &  0 &  (i-2)t
               \end{smallmatrix}\).$ Although its tropicalization
$ \(\begin{smallmatrix}
                 \one &  \one & \zero\\
                \zero &  \one &  \one\\  \one &  \zero &  \one
               \end{smallmatrix}\)$
is symmetrically nonsingular over the max-plus algebra,  it is
supertropically singular. So one could be misled to a wrong
interpretation
 without taking the supertropical structure into account.
By~\cite[Theorem~3.5]{IzhakianRowen2008Matrices},
$\QSLn$ yields a monoid under ``ghost surpasses,'' whose subset of
nonsingular elements is precisely $\QSLn$.
Our first goal then is to find the smallest natural monoid $\BQSLn$
which contains $\QSLn$, 
defined in  Definition~\ref{symmon}.  $\BQSLtwo$ is generated by
$\QSL_2$, but for $n\ge 3$, there are matrices in $\BQSLn$
 that are not factorizable, and in particular are not products of matrices from $\QSLn$ (Corollary~\ref{genmon1}).

We also investigate $\QSLn$ from within, by approaching four natural
questions:

\begin{enumerate} \eroman \dispace
  \item  What are the submonoids of $\BQSLn$ contained in $\SL_n$?
  For example, what are the maximal such submonoids?

 \item If $A \in \SL_n$, is it (2-sided) invertible in a suitable submonoid
 of $\BQSLn$?

\item We observe that  the set $\SL_n$ is too broad (not closed under multiplication), and that the set of invertible tropical
 matrices, the generalized permutation matrices, also called ``monomials,'' is too narrow (does not include $E_{i,j}$). Can we find a maximal monoid
 ``between"them?

\item Where precisely between the maximal monoid of (iii) and $\SL_n$ do we lose multiplicativity?
\end{enumerate}

Concerning (i), $\BQSLn$ contains the important submonoid generated
by the $E_{i,j}$. In Theorem~\ref{gen2} we determine this submonoid in terms of upper and lower triangular
elementary matrices.  $\QSLn$ itself has several obvious submonoids,
such as the subgroup of generalized permutation matrices, and the
upper triangular matrices. More interesting is
Example~\ref{submons}(iii), which yields a maximal nonsingular
submonoid, cf.~Theorems~\ref{thm:Jn-nb-closed} and~ \ref{permdiag}
below, built from ``strictly normal'' matrices
(cf.~Definition~\ref{norm}).

Question (ii) is perhaps more intriguing, leading to various intricacies tied to concepts from~\cite{IzhakianRowen2007SuperTropical, IzhakianRowen2008Matrices}.
One of the more intriguing aspects of tropical algebra, is that the
classical theory does not always pass to the tropical. Nonsingular matrices other than generalized permutation matrices cannot be invertible,
but we can get inversion by replacing the identity matrix by a more general version. A
 \textbf{quasi-identity matrix} $\qIl_A: =
 A A^\nb $ has many properties of the identity matrix, being
 nonsingular idempotent with  $
\per(\qIl_ A) = \one,$ even though the product of quasi-identity
matrices need not be idempotent (Example \ref{Bl2}).

 Thus, it is natural to try to write
$\BQSLn$ as the union of monoids having unit element $\qIl_A$ for
various nonsingular matrices $A$. But we have an immediate obstacle:
$\qIr_A: = A^\nb A$ might not equal $\qIl_A$,
cf.~Example~\ref{singsq}. This situation is remedied when $A$ is
\textbf{reversible}, by which we mean $\qIl_A \qIr_A \qIl_A  =
\qIr_A  \qIl_A \qIr_A$. Although this condition may look technical,
it is satisfied whenever $\qIl_A$ and $\qIr_A$  commute, which
occurs rather frequently, and is the most general condition that we
know which leads to workable submonoids, in Theorem~\ref{nonsingpr}.
It holds for $2 \times 2$ matrices when $\qIl_A\in \SL _2$,
cf.~Example \ref{singsq2}, but not for $3 \times 3$ matrices,
cf.~Example \ref{singsq3}.

In view of (iii), perhaps the most
interesting monoids arise via Definition~\ref{sln1} and Lemma~\ref{lem:id.SL3}.
Definition~\ref{sln1} introduce $\SL_n^{\one}$ as the set of normal matrices, up to
products by monomial matrices, which is shown in Theorem~\ref{thm:Jn-nb-closed} and
Theorem~\ref{permdiag} to be a maximal submonoid of~$\SL_n$.
That is, $\SL_n^{\one}$ aims to the nonsingularity property of matrices, rather than their invertability,
which provides a clear and natural approach to future study of tropical~$\GLn$.

Lemma~\ref{lem:id.SL3} defines for any
$A \in \QSLn$  a sub-semigroup of $\BQSLn$ with left unit
element $\qIl_A: = A \adj{A}, $ which contains $\qIl_A A$. This
reflects the important role of quasi-identities $\qIl_A$ in
\cite{IzhakianRowen2009Equations},
 and
``almost'' partitions~$\QSLn$ naturally into a union of submonoids.
We also consider the natural conjugation $B \mapsto A^\nb B A$ in
 \S\ref{conac}. Although some basic properties expected for conjugation fail in this setting, they do
 hold when $A$ is ``strictly normal''.

 In the last section we bring in the role of elementary matrices,
 which is rather subtle. In Lemma~\ref{sns}, we see that a Gaussian
 transformation can turn a nonsingular matrix into a singular
 matrix, which addresses point (iv). Then we show in Theorem~\ref{nonsingpr1} that although not every matrix in $\QSLn$ is itself a product of elementary matrices,
 all matrices in $\QSLn$ are equivalent with respect to multiplication by elementary
 matrices.

\section{Supertropical structures}

\subsection{Supertropical semirings and semifields} We review some basic notions from
\cite{IzhakianRowen2007SuperTropical}.

\begin{defn}\label{super1} A \textbf{supertropical semiring} is a quadruple
$R := (R, \tT, \tG, \nu)$ where  $R$ is a semiring, $\tT \subset R$
is a multiplicative submonoid, and $\tGz:= \tG \cup \{\zero\}
\subset R$ is an ordered semiring ideal, together with a map $\nu: R
\to \tGz $, satisfying $\nu^2 = \nu$ as well as the conditions:
$$a+b = \begin{cases}
   a,  &   \ \nu(a) > \nu(b), \\
\nu(a), &  \ \nu(a) = \nu(b).
\end{cases}$$
\end{defn}

 Note that $R$ contains the
``absorbing'' element $\zero$, satisfying $a+ \zero = a$ and $a
\zero = \zero a = \zero$ for all $a \in R$. The tropical theory
works for $R \setminus \{ \zero\}$, but it is convenient to assume
the existence of $ \zero $ when working with matrices.  We denote
the multiplicative unit of $R$ (and $\tT$) as $\one$.

Interpretation: The monoid $\tT$ is called the monoid of
\textbf{tangible elements}, while the elements of $\tG$ are called
\textbf{ghost elements}, and $\nu: R \to \tG \cup \{\zero\}$ is
called the \textbf{ghost map}. Intuitively, the tangible elements
correspond to the original max-plus algebra, although now $a+a =
\nu(a)$ instead of $a+a = a.$ The ideal $\tGz$ could be identified
with the max-plus algebra together with $-\infty$, but our main
tropical interest is in the tangible elements, which under the extra
conditions of Definition~\ref{superdom} below ``cover'' the ghost
elements by means of the ghost map $\nu$.

We write~$a^\nu$ for $\nu(a)$; $ a \nucong b$ stands for $a^\nu =
b^\nu$. We   define the $\nu$-order on $R$ by
$$a \nuge b  \ \dss {\Leftrightarrow}\  a^\nu \ge  b^\nu \ \text{ and }\
   a \nug b \ \dss {\Leftrightarrow} \ a^\nu > b^\nu,
  $$
The \textbf{ghost surpassing relation} on $R$ is  given by defining
$$a \lmodg  b \dss{\text{if}}   a = b + g \text{ for some } g \in \tGz.
$$

\begin{rem}\label{basicp}
We recall some basic properties concerning the ghost map, for
$a,b\in R$ and $c \in \tT$:
\begin{enumerate} \eroman \dispace
 \item $(a+b)^\nu \ge  a^\nu + b^\nu$;
 \item $(ab)^\nu = a^\nu b = ab^\nu= a^\nu b^\nu ;$
  \item $a \nuge bc^{-1}$ implies  $ac \nuge b;$
 \item  $a \nuge b$ and   $b \nuge a$ implies $a\nucong b$;
 \item    $c \lmodg a$ collapses to the standard
equality $c= a$.
\end{enumerate}
\end{rem}

\begin{defn}\label{superdom}
 A
supertropical semiring $R $ is a \textbf{supertropical semifield}
when $\tT$  is an Abelian group, $R = \tT \cup \tGz    $, and the
restriction $\nu|_\tT: \tT \to \tG$ is onto.
%
%
\end{defn}
%

\begin{example}\label{supex} 
Our main supertropical example is the \textbf{extended tropical
semiring} (cf.~\cite{zur05TropicalAlgebra}), that is, $$R = {\TR}
\cup \{- \infty \} \cup {\TR}^\nu,$$ with $\tT = {\TR}$, $\tG =
{\TR}^\nu$, where the restriction of the ghost map $\nu|_\tT: \R \to
\R^\nu$ is a natural isomorphism. Addition and multiplication are
induced respectively by the maximum and standard summation of the
real numbers~\cite{zur05TropicalAlgebra}. This supertropical
semifield extends the familiar max-plus semifield \cite{ABG}, and
serves in all of our numerical examples, in \textbf{logarithmic
notation} (in particular $\one = 0$ and $\zero = - \infty$).

\end{example}

\begin{rem} Gaubert~\cite{G.Ths},    F.~Baccelli, G.~Cohen, G.J.~Olsder, and J.P.~Quadrat~\cite{BCOQ} , and Akian,
 Gaubert, and  Guterman~\cite[Definition~4.1]{AGG} introduced   the ``symmetrized semiring''  which serves as a common generalization of~\cite{zur05TropicalAlgebra} and their earlier work.
This is a  useful semiring, which is the additive monoid  $\htRR
 := \R \cup \{-\infty\}\times \R \cup \{-\infty\}$ (two copies of the max-plus algebra, taking~$\nu$ to be the identity map), with multiplication
 given by~$(a_1,a_2)(b_1,b_2) = (a_1b_1+a_2b_2, a_1b_2+a_2b_1).$
 It follows~\cite[Remark~4.5]{AGG} that~$\tG' := \{ (a,a) : a \in \htRR\}$ is an ideal of~$\htRR $. \end{rem}

 \begin{lem} The extended tropical semiring $R$ of
\cite{zur05TropicalAlgebra} is  a homomorphic image of the
``symmetrized'' semiring $\htRR$, under the map~$ (a,b) \mapsto a+b.$
In fact, taking  $\tT' = \{(a,\zero): a \in \R\},$ one sees that $\tT' +
\tG'$ is a sub-semiring of $\htRR$ mapping onto $R$, with
$\tT'\mapsto \tT$ and $\tG'\mapsto \tG$.
\end{lem}
\begin{proof} All the verifications are easy, since $(a,a) \mapsto
a+a =  a^\nu$ and
$$(a_1,a_2)(b_1,b_2) = (a_1b_1+a_2b_2, a_1b_2+a_2b_1)\mapsto a_1b_1+a_2b_2+ a_1b_2+a_2b_1 = (a_1+b_1)(a_2+b_2).$$
\end{proof}

 Here one would identify
$\tT$ with the first component of $\htRR$, and $\tG$ with
$\tG'$. This  map is not 1:1, and there is no isomorphism from
$\htRR$ to $R$ (since the multiplicative monoid of $\htRR$
is generated by $\{ \zero\} \times R$, whereas the multiplicative
monoid of  $R$  is the group $ \R \times \R$). $\tG'$ behaves very
similarly in $\htRR$ to $\tG$ in $R$, as indicated in
\cite[Corollaries~4.18 and~4.19]{AGG}. There are some significant
differences, which justify utilizing the supertropical structure:

\begin{itemize}
\item The supertropical semiring also includes other important cases from the tropical theory,
such as  (nonarchimedean) valuations  of the Puiseux series field
$\mathbb K$, where $\tT = \mathbb K$, $\tG$ is the value group, and
$\nu$ is the valuation.

\item As noted in the introduction, linear independence of vectors is
determined in \cite{IzhakianRowen2008Matrices} in terms of the
supertropical structure, not the symmetrized structure.

\item Factorization of supertropical polynomials corresponds to decompositions of affine varieties,
\cite{IzhakianRowen2007SuperTropical}.
\end{itemize}



%
%

\section{Matrices }

%
%
%
%

In this paper we fix a supertropical semifield~$F$, and work
exclusively in the set $\MatnF$ of all $n\times n$ matrices over
$F$. We consider  $\MatnF$  as a multiplicative  monoid, with matrix
multiplication   induced from the operations on $F$. Its unit
element is the \textbf{identity matrix} $\um$ with $\one$ on the
main diagonal and whose off-diagonal entries are $\zero$. We say
that a matrix is \textbf{tangible} if its entries are all  in $\tT
\cup \{ \zero \}$, and \textbf{ghost} if its entries are all in
$\tGz$. We write $\Mat_n(\tGz)$ for the monoid of all ghost
matrices. Also we rely implicitly on Remark~\ref{basicp} throughout
the proofs of this section.

\subsection{Supertropical singularity}
\sSkip

The \textbf{tropical determinant} of a matrix $A=(a_{i,j})$  is
defined as the permanent:
  \begin{equation*}\label{eq:tropicalDet}
 \per(A) = \sum_{\pi \in \mfS_n}    \prod_{i=1}^n a_{i,\pi(i)},
\end{equation*}
where $\mfS_n$ is the set of  permutations of $\{1,\dots,n\}$.

Invertibility of matrices (in its classical sense) is
 limited in the (super)tropical setting.
\begin{rem}\label{genp} The only invertible tropical matrices are the generalized
permutation matrices, defined as the product of an invertible diagonal matrix and a permutation matrix
$P_\pi  $, such that~$(P_\pi)_{i,j}=\one$ when~$ j=\pi(i),$ and~$\zero$ otherwise.
This venerable result going back to~\cite{Rut} and~\cite{Cu}.
 (Note that $P_\pi ^{-1} = P_{\pi^{-1}}.$ See~\cite{DoO} for a rather general version of this result.)
\end{rem}
Thus, limiting nonsingularity to invertible matrices is too
restrictive for a viable matrix theory, and leads to following
definition.

\begin{defn}\label{supertropnonsing} We define a matrix~
$A\in \Mat_n(F)$  to be  \textbf{(supertropically) nonsingular} if
$\per(A) \in \tT$;  otherwise $A$ is \textbf{(supertropically)
singular} (in which case $\per(A) \in \tGz$). 
 \end{defn}
Consequently, a matrix~$A\in \Mat_n(F)$ is
singular if~$\det(A) \lmodg \zero$. This definition  does not match
the semigroup notion of regularity.

Given matrices $A  = (a_{i,j})$ and $B = (b_{i,j})$ in $\Mat_n(F)$,
we write $B \nuge A$ if $b_{i,j} \nuge a_{i,j}$ for all $i,j$, and
$B \nucong A$ if $B \nuge A$ and $B \nule A$.  The ghost surpassing
relation extends naturally to matrices, defined as
 $A \lmodg B$  if $A = B +G$
for some ghost matrix  $G \in \Mat_n(\tGz)$.  (When $A$ is tangible,
$\lmodg$ collapses to the standard equality $A= B$.)

\begin{lem}\label{lem:order1} If $A_1 \geq_\nu A_2$ and $B_1 \geq_\nu B_2$,
 then  $A_1 + B_1\geq_\nu A_2 + B_2$ and $A_1 B_1\geq_\nu A_2 B_2$. In particular, $A B\geq_\nu A$ and  $B A\geq_\nu A$
 if $B\geq_\nu I$.

 Moreover, if $A_1 \lmodg A_2$ and $B_1 \lmodg B_2$,
 then  $A_1 + B_1\lmodg A_2 + B_2$ and $A_1 B_1\lmodg A_2 B_2$, and in particular, $A B\lmodg A$ and  $B A\lmodg A$
 if $B\lmodg I$.
\end{lem}
\begin{proof} Check the components in the multiplication.
\end{proof}

\subsection{Dominant permutations}
\sSkip

\begin{df}\label{pt} A    permutation~$\pi \in S_n$ is \textbf{dominant} for~$A$
if~$\per(A) \nucong a_{1,\pi(1)}a_{2,\pi(2)}\cdots a_{n,\pi(n)}.$ A
dominant permutation~$\pi$ is \textbf{strictly dominant}
if $\prod\limits_i a_{i,\pi(i)}  \nug  \prod\limits_i a_{i,\sig(i)}
$ for any~$\sig  \neq \pi$ in~$\mfS_n$. A  strictly dominant
permutation~$\pi \in S_n$ is \textbf{uniformly dominant}
 if~$ a_{1,\pi(1)} = a_{i,\pi(i)}\ \forall i$ and~$a _{i,j}
<_\nu  a_{i,\pi(i)}\ \forall j \neq \pi(i)$.
\end{df}

Clearly the matrix $A$ is nonsingular if and only if it has a
strictly  dominant permutation,   all of whose corresponding entries
are tangible.

\begin{example}\label{pdom} The permutation $\pi$ is uniformly dominant for
the permutation matrix $P_\pi$. \end{example}

 We specify some useful classes of
matrices, to be used in the present paper, following the terminology
of \cite[\S 3]{B}. (It is only one of several usages of the
terminology ``definite'' in the literature.)
%

A \snor\ matrix (Definition~\ref{norm}) is always nonsingular, while
a normal matrix (and thus also a definite matrix) can be singular.
However, for any of these matrices we have $\per(A) \nucong \one$.

\begin{lem}\label{domin} If the permutations $\pi_t$ are uniformly dominant
for   matrices $A_t$ for $1 \le t \le \ell,$ then $\pi : =
\pi_\ell \circ \cdots \circ  \pi_1$ is uniformly dominant for $A =
A_1 \cdots A_\ell,$ and $\det (A) = \prod\limits_{t=1}\limits^\ell \det(A_t).$
\end{lem}
\begin{proof} If $\a _t = a_{i,\pi_t(i)}$ is the entry of $A_t$ for its uniformly dominant permutation $\pi _t$ (for $1 \le i \le n$),
 then $\det(A_t) =  \a _t ^n . $ On the other hand, the matrix  entries contributing to $\det(A)$  are all of the form
 $$a_{i,\pi_1(i)}a_{\pi_1(i),\pi_2\pi_1(i)}a_{\pi_2\pi_1(i),\pi_3\pi_2\pi_1(i)}\cdots
 = \a _1 \cdots \a_\ell,$$  since all other
entries are clearly less. Hence $\pi : = \pi_\ell \circ \cdots
\circ  \pi_1$ is uniformly dominant for $A$, and $\det(A) = \a_1
^n \cdots \a _\ell^n.$
\end{proof}
\noindent
Trying to weaken the hypothesis would bring us into confrontation
with Proposition~\ref{lem:perij} below.
\pSkip

%

%

%
%

We recall the basic fact {\cite[Theorem
3.5]{IzhakianRowen2008Matrices}}  that $\per(AB) \lmodg
\per(A)\per(B).$ As pointed out in \cite[Proposition 2.1.7]{G.Ths},
this result can be seen by means of transfer principles
(\cite[Theorems~3.3 and~3.4]{AGG}), and likewise is sharpened in
{\cite[Corollary~4.18]{AGG1}}; the basic idea already appears in
\cite{RS}. We shall return to this issue in Theorem~\ref{symprod}.
We denote by $S^\times$ the subset of invertible matrices (in classical sense) in a set $S$.

\begin{prop}\label{cor:A.nb.1} $\per(AB) = \per(A)  \per(B) = \per(BA) $ whenever $B \in \Mat_n(R)^\times. $
\end{prop} \begin{proof} $\per(AB) \lmodg  \per(A)  \per(B), $
and   $$\per(A) =
\per(AB B^{-1}) \lmodg  \per(AB)  \per(B^{-1}) = \per(AB)
\per(B)^{-1} .$$ Hence~$ \per(A) \per(B)\lmodg \per(AB) ,$ and
thus~$\per(AB) = \per(A)  \per(B) $. The proof that~$\per(BA) =
\per(A)  \per(B)  $ is analogous.
\end{proof}

In particular, this holds when $B$ is a
generalized permutation matrix.

\subsection{Symmetrization}
\sSkip

Following~\cite{BCOQ,Pl} and~\cite[Example~4.11]{AGG}, we define
the
 \textbf{symmetrized semiring}~$\htR $, defined to have
 the same module structure as~$ R\times R,$ but with multiplication
 $$(a_1, a_2)(a_1',a_2') = (a_1 a_1'+ a_2 a_2', a_1 a_2' + a_2
 a_1')$$
 (motivated by viewing  the first component to be in the ``positive''
  copy of~$R$ and the second component to be in the
  negative copy). Define~$R^\circ = \{ (a_1,a_2) \in R: a_1
 \nucong a_2\},$ easily seen to be an ideal of~$\htR.$
Then one defines~$(a_1,a_2) \succeq _\circ (b_1,b_2)$ in~$\htR$ if
there are~$c_i\in R$ with~$c_1 \nucong  c_2$,
 such that~$a_i = b_i +c_i$ for $i = 1,2$.
In other words,~$(a_1,a_2) = (b_1,b_2) + (c_1,c_2)$ where~$(c_1,c_2)\in R^\circ.$
%
%

\begin{lem}\label{passghost} If   $(a_1,a_2) \succeq _\circ (b_1,b_2)$, then $a_1+a_2
\lmodg b_1+b_2.$
\end{lem} \begin{proof}  If~$a_i =  b_i + c_i, i =1,2$, where~$c_1\nucong  c_2$, then~$a_1 + a_2 = b_1 +b_2 + c_1^\nu,$
  and hence~$a_1+a_2 \lmodg  b_1 +b_2$.
\end{proof}

\subsubsection{The bideterminant and symmetric singularity}

One defines  $$ \per^+(A) = \sum_{\substack{\pi \in
\mfS_n:\\\sgn(\pi) = +1}} \prod_{i=1}^n a_{i,\pi(i)},\ \ \
 \per^-(A) = \sum_{\substack{\pi \in \mfS_n:\\\sgn(\pi) = -1}} \prod_{i=1}^n a_{i,\pi(i)},
$$
 and the \textbf{bideterminant} $ \bid(A) = ( \per^+(A),\;  \per^-(A)) .$
 Note that $\per(A) = \per^+(A) +   \per^-(A).$
\pSkip

 Gaubert proved {\cite[Proposition~2.1.7]{G.Ths}}:
 \begin{prop} \label{symprod} $ \bid(AB)\succeq_\circ  \bid(A)
 \bid(B)$.\end{prop}

(This is seen most readily by means of what   Akian,
 Gaubert, and  Guterman \cite{AGG} call the \textbf{strong transfer
 principle}.)
 This result
 leads us to a more refined definition of ``nonsingular matrix.''

  \begin{defn} A matrix~$A\in\Mat_n(R)$ is \textbf{symmetrically singular} if ~$\bid(A)^\nu\in \tGz^\circ.$\end{defn}

\begin{lem} Every symmetrically singular matrix is
singular.\end{lem}
\begin{proof} $\bid( A) \in \tGz^\circ$ implies $\det( A) \in
\tG_0,$ since both components are equal.
\end{proof}

But a  singular matrix $A$  with tangible entries  can be
symmetrically nonsingular, viz.~$ A= \(\begin{smallmatrix}
                 \one &  \one & \zero\\
                \zero &  \one &  \one\\  \one &  \zero &  \one
               \end{smallmatrix}\).$

\subsection{The adjoint matrix}
\sSkip

As in the classical theory of matrices over a field, the adjoint
matrix is defined over any semiring, and has a major role in
supertropical matrix algebra, as noted in
\cite{IzhakianRowen2008Matrices,IzhakianRowen2009Equations}.

\begin{defn}\label{defn:adjoint}
The $(i, j)$-minor $A_{i,j}$ of a matrix $A = (a_{i,j})$ is obtained
by deleting the $i$-th row and the ~ $j$-th~column of $A$. The
\textbf{adjoint matrix} $\adj{A}$ of $A$ is defined as $(a'_{i,j})$,
where $a'_{i,j} = \per (A_{j,i}) $.
\end{defn}

\begin{prop}[{\cite[Proposition
4.8]{IzhakianRowen2008Matrices}}]\label{prop:adjAB}  $\adj{AB}
\lmodg \adj{B}\adj{A}$ for any $A,B \in \MatnF$.  \end{prop}

One need not have equality, as indicated in \cite[Example~
4.7]{IzhakianRowen2008Matrices}, but  by \cite[Lemma 5.7]{Niv2},
$\adj{AB} =   \adj{B}\adj{A}$  and $\adj{BA} = \adj{A} \adj{B}$ for
any generalized permutation matrix~$B$.

This result leads us to the  question  as to whether $AB$
nonsingular implies $BA$ is nonsingular. But this fails (cf.
\cite[Remark 2.12.]{IdMax}), even when $B = A^{\trn}$, the transpose
matrix:
\begin{example}\label{Bl1}  (Inspired by an idea of Guy Blachar.) \pSkip Let
$ A= \( \begin{array}{cc}
                 1 & 0 \\
                2  & 4
               \end{array}\)$, given in logarithmic notation (cf.~Example~\ref{supex}). Then  $ A^\trn = \( \begin{array}{cc}
                 1 & 2 \\
                0  & 4
               \end{array}\)$ and $A A^\trn = \( \begin{array}{cc}
                 2 & 4 \\
                4  & 8
               \end{array}\)$ which is nonsingular of determinant $10$,
               whereas $ A^\trn  A  = \( \begin{array}{cc}
                 4 & 6 \\
                6  & 8
               \end{array}\)$ is singular
               of determinant~$12^\nu$ (even symmetrically singular  of determinant~$12^\circ$).
\end{example}

%
%
%
%

\subsection{Quasi-identity matrices and the $\nb$-operation}
\sSkip

Since so few matrices are invertible, we need to replace the
identity matrix by a more general notion.

\begin{defn} A matrix $E$ is \textbf{(multiplicatively)
idempotent} if $E^2 = E.$ 
 A \textbf{quasi-identity matrix} is a
nonsingular idempotent matrix. 
\end{defn}

\begin{rem}\label{rem:uncong.QI}  The fact that a quasi-identity matrix $\qI $  is idempotent implies that  it is definite, with
its off-diagonal entries
  in $\tGz$, so this matches {\cite[Definition~4.1]{IzhakianRowen2008Matrices}}.
\end{rem}


 We define the set of all quasi-identity matrices
$$ \QI_n(F) := \{\text{Quasi-identity matrices}\} \subset \Mat_n(F), $$
each simulating the role of the identity matrix.

\begin{rem}   [{\cite[Proposition
4.17]{IzhakianRowen2008Matrices}}]\label{rem:uncong.QIadj}
$\adj{\qI} = \qI$ for every quasi-identity matrix $\qI$.
\end{rem}

On the other hand, $ \QI_n(F)$ is not a monoid.

\begin{example}\label{Bl2} $   $   \pSkip  Let
$ \qI _1 = \( \begin{array}{ccc}
                 \one & \zero & \zero \\
                \zero  &  \one &  b ^\nu  \\
                  \zero  & \zero & \one
               \end{array}\)$, $ \qI _2 = \( \begin{array}{ccc}
                 \one  &  a ^\nu & \zero \\
                \zero  &  \one & \zero \\
                  \zero  & \zero & \one
               \end{array}\)$ with  $a^\nu , b^\nu \neq \zero.$ These matrices are quasi-identities, but $ \qI _1 \qI _2 = \( \begin{array}{ccc}
                 \one &  a^\nu &   \zero\\
                \zero  &  \one &   b ^\nu\\
                  \zero  & \zero & \one
               \end{array}\)$
               is not idempotent, even though it is nonsingular.
\end{example}
\noindent


Our next task is to find  a matrix to replace the inverse. As in
classical theory, the following matrices have an important role in
supertropical matrix algebra.

\begin{defn} When $\per(A) \in \tT$, we  define the matrix $$A^\nb :=  {\per(A)}^{-1}
{\adj{A}}.$$
\end{defn}

Let us collect some information about $A^\nb.$
 We write $A^{\nb \nb}$
for $\big(A^\nb\big)^\nb$.

 \begin{lem}\label{somefac} $ $
  \begin{enumerate} \eroman
  \item  \cite[Theorem 4.9]{IzhakianRowen2008Matrices}.  $A^\nb$ is nonsingular if $A$ is nonsingular.
  \item $A^\nb = A^{-1} $ if $A$ is invertible (in view of
  Remark~\ref{genp}).
  \item   \cite[Proposition
4.17]{IzhakianRowen2009Equations}.
  $\qI ^\nb  = \qI $ for every quasi-identity matrix $\qI $.
 \item {\cite[Lemma
2.17]{IzhakianRowen2009Equations}}. \label{lem:von.N} $ \per(A)
\adj{A} \nuge  \adj{A} A \adj{A} $ for any matrix $A \in \Mat_n(F)$,
and thus  $A^\nb \nuge  A^\nb A A^\nb$ for $A$   nonsingular.
 \item   \cite[Remark 4.2]{IzhakianRowen2009Equations},
 \cite[Theorem~3.5]{Niv},
 \cite[Example~4.16]{IzhakianRowen2008Matrices}.
$A^\nbnb \lmodg A$, but $A^\nbnb \neq A$ in general.
 \item   \cite[Remark 2.18]{Niv}. $A^\nb $ is  definite
(resp.~\snor) whenever the matrix $A$ is definite (resp.~\snor).
\end{enumerate}
 \end{lem}

\begin{df}
For any nonsingular $A \in \Mat_n(F)$, we define $$\qIl_A :=   A
A^\nb  , \qquad \qIr_A :=  A^\nb  A.$$
\end{df}


The following facts are crucial.

 \begin{thm}\label{reviewadj}$ $
 \begin{enumerate} \eroman
 \item  \cite[Theorem
4.12]{IzhakianRowen2008Matrices}. $\qIl_A$ and $\qIr_A$ are
idempotent (although not necessarily equal!).
\item
  \cite[Remark~2.21]{IzhakianRowen2009Equations}.    $ \det(\qIl_A) \nucong \one $ and $\per(\qIr_A) \nucong \one$.
 \item {\cite[Theorem 4.3]{IzhakianRowen2008Matrices}}.
 $\qIl_A$ and $ \qIr_A$ are quasi-identities.
\item \cite[Corollary~4.7]{IzhakianRowen2009Equations}.  $\qIl_{A^{\nb}} =
\qIr_{A}$ and $\qIl_A =
\qIr_{A^{\nb}}.$
\end{enumerate} \end{thm}
%
%
%

In this way, $A^\nb$ could be called a ``right quasi-inverse'' with
respect to $\qIl_A$, and a ``left quasi-inverse''  with respect to
$\qIr_A$. This raises the major question, ``What is the relation
between $\qIl_A$ and $\qIr_A$?''

\begin{lem}\label{lem:Jnb.J}
$\qIl_A=\qIr_A$ for any definite matrix $A$.
\end{lem}
\begin{proof} By Lemma \ref{lem:order1}, we have
$$  \qIl_A = (A A^\nb)^2 = A  A^\nb A  A^\nb  \nuge  A A^\nb A \nuge A^\nb A = \qIr_A,$$
and, by symmetry,   $ \qIr_A    \nuge \qIl_A$, implying that $
\qIl_A    \nucong \qIr_A$. The
 off-diagonal entries of $\qIl_A$ and $\qIr_A$ are
the same (since they are all ghosts), whereas the diagonal entries
are $\one$;  thus $\qIl_A = \qIr_A.$
\end{proof}

Even when $\qIl_A \ne \qIr_A$, there is one nice situation worth
mentioning.

\begin{df}
For any nonsingular matrix $A$, we define $$ \qI_A = \qIl_A
\qIr_A \qIl_A, \qquad \widetilde{\qI_A} = \qIr_A  \qIl_A \qIr_A .$$
  We say that $A$ is \textbf{reversible} if  $\qI_A = \widetilde{\qI_A} .$
\end{df}

\begin{lem} $\widetilde{\qI_A} =
\qI_{A^\nb}.$\end{lem}\begin{proof}$ \qIr_A  \qIl_A \qIr_A  =
\qIl_{A^\nb} \qIr_{A^\nb} \qIl_{A^\nb} = \qI_{A^\nb}.$
\end{proof}

Reversibility gains interest from the following result.

\begin{prop}\label{pushdown1} If  $A$ is
reversible and $\qI_A$ is nonsingular, then $\qI_A$ is a
quasi-identity.
  \end{prop}
  \begin{proof}    $\qI_A$ is idempotent since
   $$\qI_A  ^2 =  \qIl_A \qIr_A \qIl_A \qIr_A \qIl_A =
 (\qIr_A \qIl_A \qIr_A )\qIr_A \qIl_A  =  (\qIr_A \qIl_A \qIr_A )
 \qIl_A = \qIl_A \qIr_A \qIl_A \qIl_A = \qI_A   .$$

  and $\per(\qI_A) = \per( \qIl_A)\per(
\qIr_A)\per( \qIl_A) \nucong \one^3 = \one$ by Theorem~\ref{reviewadj}.
\end{proof}
%

\begin{prop}\label{singsq22} If  $\qIl_A \qIr_A  = \qIr_A \qIl_A,$ then
$A$ is reversible, and $\qI_A = \qIl_A \qIr_A $.
  \end{prop}
\begin{proof}  $\qIl_A \qIr_A \qIl_A =  \qIr_A \qIl_A \qIl_A =  \qIr_A
\qIl_A= \qIl_A \qIr_A = \qIl_A \qIr_A \qIr_A= \qIr_A \qIl_A \qIr_A
.$
\end{proof}

But here is an example, obtained by modifying an example from
\cite{zur05TropicalAlgebra} showing the complexity of the situation
in general.

\begin{example}\label{singsq} (logarithmic notation) \pSkip Take  $ A= \(\begin{matrix}
                 -1 & -1 \\
                0  & 1
               \end{matrix}\)$
whose determinant is $0 (= \one),$ whereas $    A^2= \(
\begin{matrix}
                 -1 & 0 \\
                1  & 2
               \end{matrix}\)$ is singular.
We have  $A^\nb   = \( \begin{matrix} 1 & -1 \\
                0  & -1 \end{matrix}\)$, and the quasi-identity matrices
 $$   \qIl_A = A A^\nb = \( \begin{matrix}
                 0 & (-2) ^\nu \\
                1^\nu  & 0
               \end{matrix}\),\qquad \qIr_A = A^\nb  A = \( \begin{matrix}
                 0 & 0 ^\nu \\
                -1 ^\nu  & 0
               \end{matrix}\). $$
               We see that $$\qIl_A \qIr_A =  \( \begin{matrix}
                 0 & 0 ^\nu \\
                1 ^\nu & 1^\nu
               \end{matrix}\) \neq    \( \begin{matrix}
                1^\nu  & 0 ^\nu \\
                1^\nu  & 0
               \end{matrix}\) = \qIr_A \qIl_A.$$
               \end{example}

Here is the general situation for $2 \times 2$ matrices, in
algebraic notation.
\begin{example}\label{singsq2} \pSkip Take  $ A= \(\begin{matrix}
                 a & b \\
                c & d
               \end{matrix}\)$
whose determinant $ ad+bc $ is $ \one.$ Then~$    A^\nb= \(
\begin{matrix}
               d & b \\
                c & a
               \end{matrix}\),$ and we get the quasi-identity matrices
 $$ \  \qIl_A = A A^\nb = \( \begin{matrix}
                 \one & (ab) ^\nu \\
                (cd) ^\nu  &  \one
               \end{matrix}\)   \ne   \qIr_A = A^\nb  A = \( \begin{matrix}
                 \one & (bd) ^\nu \\
                (ac) ^\nu  &  \one
               \end{matrix}\). $$
               We see that $$ \text{$\ \qIl_A \qIr_A =  \( \begin{matrix}
                 \one + (a^2bc ) ^\nu  & b^\nu (a+d) \\
                c^\nu (a+d)  &  \one +(bcd^2)^\nu
               \end{matrix}\)$ whereas    $\ \qIr_A \qIl_A =  \( \begin{matrix}
         \one +(bcd^2)^\nu
              & b^\nu (a+d) \\     c^\nu (a+d) & \one  + (a^2bc ) ^\nu
               \end{matrix}\) = (\qIl_A \qIr_A )^\nb,$}$$
               but when either is nonsingular, then they are both equal to $ \( \begin{matrix}
         \one
              & b^\nu (a+d) \\     c^\nu (a+d) & \one
               \end{matrix}\),$ implying $A$ is reversible in this case.
               \end{example}

This raises hope that the theory works when we only encounter
tangible matrices, but a troublesome example exists for $3 \times 3$
matrices.
\begin{example}\label{singsq3} (logarithmic notation, where $-$ denotes $-\infty$) \pSkip Take  $ A= \(\begin{matrix}
                 - & 5 & 0 \\
                0  & -  & - \\ - &     0  & -
               \end{matrix}\)$
whose determinant is $0 = \one.$ Then
 $ A^\nb = \(\begin{matrix}
                 - &   0 & -\\
                 -  & - & 0  \\     0  & - & 5
               \end{matrix}\),$ so

 $$   \qIl_A = A A^\nb = \(\begin{matrix}
                    0  & - & 5^\nu \\ - &   0 & -\\
                 -  & - & 0
               \end{matrix}\), \qquad \qIr_A = A^\nb A = \(\begin{matrix}
                    0  & - & -\\ - &   0 & -\\
                 -  & 5^\nu & 0
               \end{matrix}\),$$
which are both definite (and would be strictly normal if we took $-5$ instead
of $5$). But
              $$ \qIl_A \qIr_A =  \(\begin{matrix}
                    0  & 10^\nu & 5^\nu \\ - &   0 & -\\
                 -  &  5^\nu  & 0
               \end{matrix}\)   \qquad \neq  \qquad  \(\begin{matrix}
                    0  & - & 5^\nu \\ - &   0 & -\\
                 -  &  5^\nu  & 0
               \end{matrix}\)   = \qIr_A \qIl_A.$$
               Furthermore, $ \qIl_A \qIr_A  $ is idempotent and nonsingular,
              and thus a quasi-identity, so $\qIl_A \qIr_A  = (\qIl_A \qIr_A )^\nb$ which does not equal
               ${\qIr_A }^\nb {\qIl_A }^\nb = \qIr_A   \qIl_A.$  \end{example}

 The quasi-identities $\qIl_A$ and $ \qIr_A$ always   satisfy a nice relation in the  $2 \times 2$  case. We say that
 $2 \times 2$ matrices $\mathcal I = \(\begin{matrix}
                \one & u ^\nu \\
                v ^\nu & \one
               \end{matrix}\)$  and  $\mathcal
{I'}=  \(\begin{matrix}
                \one & {u'} ^\nu \\
                {v'} ^\nu & \one
               \end{matrix}\)$ (in algebraic notation) are \textbf{paired} if   $u^\nu v^\nu = {u'}^\nu {v'}^\nu.$

\begin{lem}\label{singsq31}  For any $2 \times 2$ matrix $A$ of determinant $\one,$ the
quasi-identities $\qIl_A$   and $\qIr_A$ are paired. Conversely, if
$F$ is closed under square roots then, given paired quasi-identity
matrices $\mathcal I$ and $\mathcal I',$ there is  a $2 \times 2$
matrix $A$ of determinant $\one,$ such that $\mathcal I = \qIl_A$
and $\mathcal I'= \qIr_A$.\end{lem}\begin{proof}  After a
permutation, we may write $ A= \(\begin{matrix}
                a & b \\
                c  & a^{-1}
               \end{matrix}\)$ with $bc <_{\nu} \one$.
                Then  $$\qIl_A = \(\begin{matrix}
                \one & (ab) ^\nu \\
                (a^{-1}c) ^\nu & \one
               \end{matrix}\) \dss{\text{and}}
\qIr_A= \(\begin{matrix}
                \one &  (a^{-1}b) ^\nu \\
                (ac) ^\nu & \one
               \end{matrix}\)$$  are paired since~$(ab)(a^{-1}c) = bc = (a^{-1}b)(ac).$
               Conversely, given $uv' = u'v$ we take $b =
               \sqrt{uu'},$ $c = \sqrt{vv'},$ and $a = \sqrt {\frac
               u{u'}}$ to get $ab = u,$ $a^{-1}b = u',$ $ac = \sqrt
               {\frac{uvv'}{u'}} = \sqrt{v^2} = v,$  and $a^{-1} c = v'$.

\end{proof}

%
%

 \begin{lem}[{\cite[Lemma
2.17]{IzhakianRowen2009Equations}}] \label{lem:von.N1}
 $A^\nb \le A^\nb A A^\nb$ for any matrix $A \in \MatnF$.
 \end{lem}


%
%

%
%
%

\begin{defn} A matrix $A$ is
\textbf{$\nb$-regular}
 if $A = A A^\nb A$.
 \end{defn}

\begin{example} $A A^\nb A = \qIl_A A = A \qIr_A$ is $\nb$-regular (but not necessarily reversible, nor nonsingular). Every
quasi-identity matrix is $\nb$-regular as well as reversible.
\end{example}

Since $A A^\nb A$ shares many properties with $A$ (for example,
yielding the same quasi-identities $\qIl_A$ and~$\qIr_A$ and other
properties concerning solutions of equations in
~\cite{IzhakianRowen2009Equations}), the passage to $A A^\nb A$
is a closure operation which is of particular
interest to us.

%


\section{Special linear supertropical matrices}

As stated earlier, our main objective is to pinpoint the most viable tropical
version of $\SLn$. The obvious attempt is the set $$ \QSLn(F) := \{
A \in \Mat_n(F) \ds : \per(A) = \one \}$$
 of matrices with supertropical determinant $\one$, which we call \textbf{special linear matrices}.

\subsection{The monoid generated by $ \QSLn(F)$} \sSkip

$ \QSLn(F)$ is not a monoid, as noted in Equation~\eqref{badprod}.
Thus, we would like to determine the monoid generated by $\QSLn(F)$,
as well as the submonoids of $\QSLn(F)$.

\begin{rem} \label{factorin} For the matrices $A
\in  \QSLn$ and $B\in \QSLn^\times$, we have   $AB, BA \in \QSLn,$ by Proposition~\ref{cor:A.nb.1}.
\end{rem}

Thus, any difficulty  involves noninvertible matrices of $
\QSLn(F).$  The following observation ties this discussion to
definite matrices.

\begin{lem} \label{factorout}$ $ \begin{enumerate}\eroman \dispace
 \item Any nonsingular matrix $A$ is the product $P A_1$ of a  generalized
 permutation matrix~$P$ with a definite
matrix $A_1 $.

 \item  Any   matrix $A$ of $\QSLn(F)$ is the product $PA_1$ of
 a  generalized
 permutation matrix $P \in \QSLn(F)$ with a  definite
matrix $A_1 \in \QSLn(F)$. Likewise we can write $A = A_2 Q$ for a
generalized permutation matrix $Q$ in $\QSLn$ and $A_2$ a definite
matrix.
\end{enumerate}
\end{lem}
\begin{proof}   Multiplying by a permutation matrix puts the
dominant permutation of~$A$ on the diagonal, which we can make
definite by multiplying by a diagonal matrix. If~$A \in \QSLn(F)$
then~$A_1 \in \QSLn(F)$, in view of Proposition~\ref{cor:A.nb.1}.
\end{proof}

%

The point of this lemma is that the process of passing a matrix of
$\QSLn$ to definite form takes place entirely in $\QSLn$, so the
results of \cite{Niv} are applicable in this paper, as we shall see.

\begin{defn}\label{symmon}
 $$ \BQSLn(F) :=
\big\{ A \in \Mat_n(F) \ds :  \det(A) \lmodg \one\}.$$ We write
$\QSLn$ and $\BQSLn$ for $\QSLn(F)$ and $\BQSLn(F)$ respectively, when $F$ is
clear from the context.
\end{defn}

 In the spirit of \cite[Proposition
3.9]{IzhakianRowen2008Matrices}, but using symmetrization, we can
turn to $\succeq_\circ$, and define:

\begin{defn}\label{genmon}
 $$ \QSLn(F)_\circ := \QSLn(F)\cup
\big\{ A \in \BQSLn(F) : \  \bid(A)= (\alpha, \beta) \
\text{where} \  \alpha = \one^\nu >_\nu   \beta\text{ or } \beta = \one^\nu   >_\nu \alpha
\}.$$
\end{defn}

 Thus $\QSLn(F) \subset  \QSLn(F)_\circ \subset \BQSLn(F)$, so we could increase the scope of
 the theory by considering $\QSLn(F)_\circ $ instead of $\QSLn(F)$.

%

Here is a generic sort of example.

\begin{example} Consider two rank 1
matrices $  \(\begin{matrix}
                a  & \zero \\
              b  & \zero
               \end{matrix}\)$ and $  \(\begin{matrix}
                c   & d \\
              \zero  & \zero
               \end{matrix}\)$.
               Their product is $  \(\begin{matrix}
                ac  & ad \\
              bc & bd
               \end{matrix}\)$
               whose bideterminant is $(abcd, abcd)\in F^\circ.$

               Although the first two matrices are singular, they
               ``explain'' the following modification:
               The product of the matrices
               $  \(\begin{matrix}
                a  & \zero \\
              b  & a^{-1}
               \end{matrix}\)$ and $  \(\begin{matrix}
                c   & d \\
              \zero  & c^{-1}
               \end{matrix}\)$, both from $ \QSLn(F)$, is
            $  \(\begin{matrix}
                ac  & ad \\
              bc & bd + a^{-1}c^{-1}
               \end{matrix},\)$
               which is $  \(\begin{matrix}
                ac  & ad \\
              bc & bd
               \end{matrix}\)$ when $abcd > \one$.
Put another way, given any   $u,v,u',v'\in F$ satisfying $uv' =
u'v$, we can find two matrices whose product is $  \(\begin{matrix}
                u  & u' \\
              v & v'
               \end{matrix}\)$,
               namely take $a = \one,$ $c = u,$ $d = u'$, and $b =
               \frac{v}{u}.$
Thus, every $2 \times 2$ matrix in $\BQSL_2(F) $ is a product of two
matrices in $ \QSL_2(F)$.
\end{example}

This yields:

\begin{prop} $\BQSL_2(F) $ is the submonoid of matrices generated by
$ \QSL_2(F)$.
\end{prop}
\begin{proof}  The key computation is
$  \(\begin{matrix}
                \one  & b\\
               a  b  ^{-1} & a
               \end{matrix}\) = \(\begin{matrix}
                \one  & \zero \\
               a  b  ^{-1} &  \one
               \end{matrix}\)  \(\begin{matrix}
                \one  &
              b  \\ \zero & \one
               \end{matrix}\)$
              when $a >_{\nu} \one,$  a special case of the previous example.
\end{proof}

On the other hand, for larger $n$, we have room for obstructions.

\begin{example}\label{genmon11} For $n \ge 3,$ suppose $A$ has the form
 $$  \(\begin{matrix}
                a_{1,1}  & a_{1,2} & \zero & \zero & \dots & \zero \\
             \zero & a_{2,2}  & a_{2,3} & \zero   & \dots & \zero\\
            \zero &  \zero & a_{3,3}  & a_{3,4}  & \dots & \zero
            \\  \vdots  & & & \ddots & \ddots &  \vdots
\\
             \zero &   \dots & \zero & \zero &  a_{n-1,n-1} &   a_{n-1,n}
 \\
              a_{n,1} &\zero &   \dots & \zero & \zero &    a_{n,n}
              \end{matrix}\).$$
               Then $A$ cannot be factored into $A_1A_2$ unless one
               of the $A_i$ is invertible.
\end{example}

More generally, we have:

\begin{prop} \label{nonfact}  {\cite[Proposition 3.2]{Niv2}} Suppose $\pi, \sigma \in S_n$ such that there exists an integer $0 < t < \frac n2$  for
which,  for all $i$,  $\pi(i) \equiv \sigma(i) + t \pmod n.$ Then
any $n \times n$ matrix $A = \sum _{i=1}^n (a_{i,\pi(i)}e_{i,\pi(i)}
+ a_{i,\sigma(i)}e_{i,\sigma(i)})$ (with invertible coefficients
$a_{i,\pi(i)}, a_{i,\sigma(i)}$)   is not factorizable.
\end{prop}

%
%

For $n=3$ this example is not so bad, since $A$ has no odd
permutations contributing to the determinant. But for $n$ even, $A$
has one odd permutation and one even permutation which contribute.

\begin{cor}\label{genmon1}
For even $n\ge 4$, $\BQSL_n(F) $ is not a product of elements of
$\QSL_n(F) $.\end{cor}
\begin{proof} The permutation $(1 \ 2 \ \cdots \ n)$ is odd, and so we
get an element of  $\BQSL_n(F) \setminus \QSL_n(F)$ which is not factorizable, and in
particular is not a product of elements of $\QSL_n(F) $.
\end{proof}

%
%
%
%

%

\subsection{Nonsingular submonoids} \sSkip

 Although~ $\QSLn$  is not a monoid, it does have interesting
 submonoids.
 \begin{df}
A matrix monoid is \textbf{nonsingular} if it consists of
nonsingular matrices.

 A subset $S \subset \Mat_n(F)$ is
\textbf{$\nb$-closed} if
  $A^\nb \in S $ for all $A \in S$.
 \end{df}
Geometric and combinatorial characterizations of nonsingular
tropical matrix monoids are provided in \cite{GIMM,IJK2},
where these monoids admit nontrivial (universal) semigroup identities \cite{IdMax}.
 Most of the sets we consider are $\nb$-closed.

\begin{example}\label{submons}$ $ \begin{enumerate}\eroman
\dispace
 \item The set of generalized
permutation matrices in $\QSLn$ is a  nonsingular subgroup   of $\INV$ (with
unit element~$I$).\item The upper triangular matrices of $\QSLn$ are
a submonoid (with unit element $I$).
\item
 If $A$ is \snor, then
the monoid generated  by  $A$  is nonsingular. (Indeed, $A^k$ is
nonsingular for any $k  < n$, which means that there is only one way
of getting a maximal diagonal entry in any power of $A$, which is
by taking a power of $a_{i,i} = \one$, and the non-diagonal entries
will be smaller.)
\end{enumerate}
\end{example}

We continue with (iii), and appeal to a more restricted version of
$\QSLn$.

\begin{df}

$\tJn $ denotes the set of all \snor\ $n\times n$ matrices in
$\Mat_n(F)$.
\end{df}

\begin{lem}\label{prop:Kn} $\tJn$ is a nonsingular $\nb$-closed monoid, also closed under transpose.
\end{lem}
\begin{proof} A straightforward verification, using Lemma~\ref{somefac}(vi).
\end{proof}

%
%


\subsubsection{The $\one$-special  linear monoid}
\sSkip

We enlarge the monoid $\tJn$   via the left and right action of permutation matrices.

\begin{df}\label{sln1} Given a set $S$, we define its \textbf{permutation
closure} to be $$\text{$\{ P  J Q : J \in S$  and $P ,Q $ are
permutation matrices$\}.$}$$
 The~\textbf{$\one$-special linear monoid}~$\SLnI$ is  the permutation closure of the monoid~$\tJn$  of  \snor\ matrices.
\end{df}

\begin{remark} A matrix $A$ of $\Mat_n(F)$ is in $\SLnI$ if and only if  $A = (a_{i,j})$ has a uniformly dominant
permutation $\pi$ with $a_{i,\pi(i)} = \one$ for all~$i$.
\end{remark}

\begin{thm}\label{thm:Jn-nb-closed} $\SLnI$ is a $\nb$-closed submonoid of $\QSLn$.
\end{thm}
\begin{proof}
Write~$B = A_1 A_2$ where~$A_1 =  P_{\pi_1} J_1 Q_{\pi_2},\ A_2 =
P_{\pi_3}  J_2 Q_{\pi_4}\in\SLnI,$ with~$J_1, J_2 $ \snor.
Thus ~$B $ is a product of matrices with respective
uniformly dominant permutations~$\pi_1, \id, \pi_2, \pi_3, \id,
\pi_4$, and we see from Lemma~\ref{domin} that~$\tau = \pi_1
 \id  \pi_2  \pi_3 \id \pi_4$ is uniformly dominant for~$B$, in which~$b_{i, \tau(i)} = \one,\ \forall i$. Hence,~$B\in \SLnI$, and we have proved that~$\SLnI$ is a monoid.

  By Lemma~\ref{prop:Kn} and Proposition~\ref{cor:A.nb.1} it follows that  $$ A^\nb = (P_{\pi_1}  J Q_{\pi_2})^\nb = Q_{\pi_2}^\nb J^\nb P_{\pi_1}^\nb =
Q_{\pi_2}^{-1} J^\nb P_{\pi_1}^{-1},$$ for every $A \in \SLnI,$ and
thus  $\SLnI$ is $\nb$-closed.
\end{proof}

Assume that the semifield $F$ is \textbf{dense} in the sense that if
$a >b$ in $F$ then there is $u<\one$ such that $ua>b.$ The next lemma and proposition show why   matrices not in $\SLnI$ and
permutation matrices do not mix well.
\begin{lem} \label{lem:perij0} Suppose  $ A= \(\begin{matrix}
                a  & b \\
              c  & d
               \end{matrix}\)$, or  $ A= \(\begin{matrix}
                b  & a \\
              d  & c  \end{matrix}\)$
where $a  < \one < d$ and $bc<ad$.
There exists $U\in\SLnI$ such that $U^t AU$ is symmetrically singular.
\end{lem}
\begin{proof} There is $u$ with $
             a < u <\one$ for which $du^2  >
               a$, and thus, taking $U = \(\begin{matrix}
           \one   &   \zero \\  u & \one
               \end{matrix}\),$ and $U^t = \(\begin{matrix}
           \one   &  u \\   \zero & \one
               \end{matrix}\),$  in the first case
               we have $$U^t AU= U^t(AU) = U^t \(\begin{matrix}
                a +bu & b \\
              c +du & d
               \end{matrix}\)= \(\begin{matrix}
                 (b+c)u +  du^2 & b +du \\
              c +du & d
               \end{matrix}\) $$
               whose bideterminant is $((b+c)du +  d^2u^2,(b+c)du +  d^2u^2).$
               The second case works in the same way.
\end{proof}

\begin{prop} \label{lem:perij}  Suppose $M\in\QSLn$ but not in $ \SLnI$. Then there exists  a matrix $U\in\SLnI$ such that
either~$MUM$ is symmetrically
singular with $U$  a permutation matrix,
 or $U^tMU$ is symmetrically
singular with $U$ a strictly normal matrix.
\end{prop}
\begin{proof} We write the dominant track of~$M=(a_{i,j})$
as $a_{i_1,\pi(i_1)}\le a_{i_2,\pi(i_2)} \le \dots \le a_{i_n,\pi(i_n)}.$ Reordering
the indices we may assume that $a_{1,\pi(1)} \le a_{2,\pi(2)} \le
\dots \le a_{n,\pi(n)}.$

First assume that  all the~$a_{i,\pi(i)}=\one.$ By hypothesis~$a_{i,\pi(j)} \ge \one$
for some~$j\ne i$. Consider the~$2 \times 2$ matrix
 $B:= \(\begin{matrix} a_{i,\pi(i)} & a_{i,\pi(j)} \\
               a_{j,\pi(i)}   &  a_{j,\pi(j)}
               \end{matrix}\) = \(\begin{matrix}
                \one  & a_{i,\pi(j)} \\
               a_{j,\pi(i)}   &  \one
               \end{matrix}\).$
                Since $M$ is nonsingular, we must have $ a_{i,\pi(j)} a_{j,\pi(i)}
               <\one,$ so~$ a_{j,\pi(i)}
               <\one.$
               Let $P$ be the $2 \times 2$ permutation matrix $\(\begin{matrix} \zero & \one \\
            \one   &  \zero \end{matrix}\),$ so $PB =\(\begin{matrix}
               a_{j,\pi(i)}   &  \one    \\  \one  & a_{i,\pi(j)}
               \end{matrix}\),$ and
   $$BPB =  \(\begin{matrix} a_{i,\pi(j)} &  a_{i,\pi(j)}^2\\
                \one  & a_{i,\pi(j)} \\

               \end{matrix}\),$$
    which is symmetrically singular.
  Extending $P$ to the
               $n\times n$
               permutation matrix $U$ corresponding to the transposition $(i\ j)$, yields
               $MUM$  symmetrically singular.

Thus we may assume that some $a_{i,\pi(i)}<\one,$ and some
$a_{j,\pi(j)}>\one.$ Then Lemma~\ref{lem:perij0} is applicable
taking $A=B.$ Therefore,
               $U^t MU$  is symmetrically singular, taking $U$ to be the elementary matrix $E_{i,j}(u)$, where~$u$ as in Lemma~\ref{lem:perij0}.
\end{proof}

\begin{thm}\label{permdiag} The monoid $\SLnI$ is a maximal nonsingular
submonoid of $\QSLn$.
\end{thm}
\begin{proof} $\SLnI$ is a  nonsingular
submonoid of $\QSLn$ by Theorem~\ref{thm:Jn-nb-closed}.  $\SLnI$ is
maximal nonsingular by Proposition~\ref{lem:perij}, since for  $M\in\QSLn$ but not in $\SLnI$, there exists $U\in\SLnI$, such that either $U^tMU\notin\QSLn$ or $MUM \notin\QSLn$.
\end{proof}

%
%

But $\SLnI$ is not the only maximal nonsingular submonoid of
$\QSLn$, since it does not contain the other monoids
of~Example~\ref{submons}.

%
%
%

%
%

%
%
%

\subsection{Submonoids of $ \BQSLn(F)$} \sSkip

 Define $T^u$ to be the set of products of $E_{i,j}$ with  $i <j$, and
$T^\ell$ to be the set of products of $E_{i,j}$ with  $i >j$, the
respective sets of \textbf{upper} and \textbf{lower} triangular
matrices. These are both monoids, and we want to consider $T^\ell
T^u$. Toward this objective, we commute elements of $T^u$ and
$T^\ell.$

The following argument, based on the Steinberg relations of the $
E_{i,j}.$  We write $ E_{i,j}(a)$ for $I +
a e_{i,j}.$

\begin{lem}\label{rearrange} $ $
\begin{enumerate} \eroman
\item $ E_{i,j}(a) E_{k,\ell}(b) =  E_{k,\ell}(b) E_{i,j}(a)\quad \text{for}\quad
i\ne \ell, j \ne k.$

\item  $ E_{i,j}(a) E_{j,i}(b) =  E_{j,i}(b)  E_{i,j}(a) \ \ \text{ for} \ ab <\one.$
\item $ E_{i,j}(a) E_{j,k}(b) =  \begin{cases} E_{i,k}(ab) E_{j,k}(b)  E_{i,j}(a) &  \text{for}\
k<i<j, \\
 E_{j,k}(b)  E_{i,j}(a)E_{i,k}(ab) &  \text{for}\ i<k<j.
\end{cases}$
\end{enumerate}
\end{lem}
\begin{proof}
Direct computation.
 \end{proof}

\begin{thm}\label{gen2} Any product of $E_{i,j}$ matrices contained in $\QSLn$ is in $T^\ell T^u.$
\end{thm}

 \begin{proof} The relations in Lemma~\ref{rearrange} enable us to
 move all $ E_{i,j}$ for $i<j$ to the right, so by induction we can
 rearrange any product of elementary matrices to a product $AB$
 where $A \in  T^\ell $ and $B \in  T^u$.
 \end{proof}

Of course, in our situation, this is a submonoid of $ \BQSLn(F)$.

\begin{example} Let~$E$ denote the set of~$2\times 2$  definite matrices, and~$E_{\operatorname{sing}}$ denote the set of matrices of the form~$c\(\begin{matrix}
               ab & a \\
              b & \one
               \end{matrix}\)$ or~$c\(\begin{matrix}
               \one & a \\
              b & ab
               \end{matrix}\)\text{ such that~}ab,c\geq \one.$
The  monoid generated by $E_{1,2}$ and $E_{2,1}$ is $E
\cup E_{\operatorname{sing}}$:

 $$ \(\begin{matrix}
               \one + a_1 a_2 & a_1 \\
              a_2 & \one
               \end{matrix}\)\(\begin{matrix}
               \one  & b_1 \\
             b_2 & \one + b_1 b_2
               \end{matrix}\)=  \(\begin{matrix}
               \one + a_1 (a_2 + b_2)  & a_1 + b_1 +   a_1 b_1(a_2 + b_2) \\
              a_2+b_2 &    \one +  b_1(a_2+ b_2)
               \end{matrix}\). $$

\end{example}

The LU-factorization, attributed to Turing~\cite{Tu}, is one of the
pillars of classical matrix algebra, cf.~\cite[Theorems 1E and
1F]{Str}. Theorem~\ref{gen2} gives us the LU-factorization for nonsingular
products of elementary matrices. We also recall
that~ all nonsingular definite $2\times 2$ matrices have an~LU-factorization
(see~\cite[Example~2.9]{Niv2}), but this fails already in
 the~$3\times 3$ case, as seen by Proposition~\ref{nonfact}.
  Nevertheless, by~\cite[Corollary~6.6]{Niv2}, we
have~$\{A^\nabla: A\in \SL_n\}\subseteq T^\ell T^u .$ We do not know
if this inclusion is strict.

\begin{conjecture} If~$B$ is a non-triangular definite matrix in~$T^\ell T^u$, then~$B\in\{A^\nabla: \ A\in \SL_n\}.$\end{conjecture}

It has been  recently proved in~\cite{TTNN} that the monoid
generated by Jacobi matrices $E_{i,i\pm 1},$ is the set of tropical totally nonnegative
matrices (defined by means of sign-singularity and dominant
permutation parity) with  non-$\zero$  determinant. However, the question of
what is generated by all the $E_{i,j}$ remains open.

We further study the action of these tropical nonsingular noninvertible elementary matrices in  \S\ref{5}.

\subsection{Semigroup unions in $\BQSLn$}\label{part0} \sSkip

Our objective here is to carve $\BQSLn$ into monoids, each of which
has a multiplicative unit $\qI$, where $\qI$ is a quasi-identity.
Although we cannot quite do this, the process works for
$\nb$-regular matrices.


\begin{defn}\label{goodmon} For any $A\in \Mat_n(F)$ with $\per(A) \neq \zero$: \begin{enumerate}\eroman \dispace
  \item $\SLmonlA = \{ B \in \BQSLn : \qIl_A B = B\}$;
    \item $\SLmonrA = \{B \in \BQSLn : B\qIr_A = B\}$.
      \item $\SLmonA  =  \{ B \in  \BQSLn : \qI_A B = B\qI_A =  B\}$.
\end{enumerate}
\end{defn}

In particular, for a  quasi-identity  $\qI$, \\
 $$\SLmonl  = \{ B \in\BQSLn : \qI B = B\}\ \ \ \text{ and }\ \ \
  \SLmonr  = \{B \in \BQSLn : B\qI = B\}.
$$

\begin{lem} If $A$ is $\nb$-regular, then $A \in\SLmonlA \cap \SLmonrA $. \end{lem}
\begin{proof} $A = A A^\nb A = \qIl_A A = A \qIr_A $.
\end{proof}

%

\begin{lem}\label{lem:id.SL3} $\SLmonA$ is
a sub-semigroup of $\BQSLn$ with left unit element $\qIl_A $ and
right unit element $\qIr_A $.
\end{lem}
\begin{proof} First note that if $\qIl_A B _1 = B_1$ then for any
$B_2$ we have $\qIl_A (B _1 B_2) = B_1B_2.$ Thus $\SLmonlA$ is
closed under multiplication on the right by any matrix. In
particular,  $\SLmonA$ is
a sub-semigroup of $\BQSLn$. The other assertion holds since
$\qIl_A $ and  $\qIr_A $ are idempotent.
\end{proof}

This provides the intriguing situation in which we have a natural
semigroup with left and right identities  which could be unequal. 
%
The situation is better when $A$ is reversible.

\begin{thm}\label{nonsingpr} Every reversible element  $A$ of $
\QSLn$ defines   a submonoid $\SLmonA$   with unique unit element~
$\qI_A$, and which contains $\qI_A A$. The union of these submonoids
contains every reversible $\nb$-regular element of $\QSLn$, and in
particular, every quasi-identity matrix.
\end{thm}
\begin{proof}
 If~$A$ is reversible, then~$\qI_A \in \SLmonA$ is the (unique) unit element, in view of
Proposition~\ref{pushdown1}, so  $\SLmonA$ is
a monoid.
Furthermore,~$\qI_A A = A A^\nb A \in\SLmonA$. The last assertion follows at once.
\end{proof}

\section{The conjugate action}\label{conac}

For any nonsingular matrix $A$ and any matrix $B$, we define
$$\mcong {B}{ A} = {A^\nb} B A.$$ This is the closest we have to
conjugation by supertropical matrices. (Note that $\mcong{I}A =
A^\nb I A  = \qIr_A.$)


We continue with an example  of a nonsingular matrix having a
singular conjugate.

\begin{example}\label{exmp:2x2} Take~$A = \( \begin{array}{cc}
           \al &  \one \\
           \one & \bt
         \end{array} \), \
         B = \( \begin{array}{cc}
           x &  z \\
           w & y
         \end{array}
\), $ where~$x \nug y$,\ $xy \nug zw\nuge \zero $, and~$\al, \bt \nul
\one$ such that~$\al \bt \nug \frac y x$. Then
\begin{align*} A^\nb B A & =
 \( \begin{array}{cc}
           x \al \bt +z \bt+  w\al +  y   &  x \bt + z \bt^2  + y \bt  + w\\
           x \al  + w \al ^2 + y \al +   z   & z \bt + y \al \bt  +w \al +  x
         \end{array}
\)= \( \begin{array}{cc}
           x \al \bt +z \bt +  w\al    &  x \bt + z \bt^2   + w\\
           x \al  + w \al^2 +   z   & z \bt + w \al +  x
         \end{array}
\),\end{align*} for which
$$\per(A^\nb B A) = (w x \al +w^2 \al ^2+x z \bt +x^2 \al  \bt +w x \al ^2 \bt +z^2 \bt ^2+x z \al  \bt ^2)^\nu,$$
since~$x^2 \al \bt \nug xy \nuge wz \nug w z \al  \bt  \nug   w z
\al^2  \bt^2.$ Thus~$A^\nb B A$ is singular.
Obviously, this holds for any nonsingular matrix $B$ with $y =
x^{-1}$, namely when $\per(B) = \one. $
\end{example}

Given a nonempty set $S \subset \Mat_n(F)$ of matrices and a matrix
$A$ with $\qI_A \in S$, we write
$$ \mcong S A  = \big\{ A^\nb B A  \  :  \   B \in S \}.$$

%

If $S$ is a monoid, then $ \mcong S A $ also is a monoid. But when
$A\in \QSLn$ is not invertible, we get into difficulties, even in
the $2 \times 2$  case.

\begin{example} \label{singsq41}(logarithmic notation) \pSkip $B = \left(\begin{array}{cc}
 0 & {-\infty}   \\
 {1} & 0
\end{array}\right)$ is definite, and  $A =  \left(\begin{array}{cc}
 0 & {5^\nu}   \\
 {-\infty} & 0
\end{array}\right)$ is a quasi-identity matrix, but $BAB = \left(\begin{array}{cc}
 6^\nu & 5^\nu   \\
7^\nu & 6^\nu
\end{array}\right)$ is singular.
\end{example}

 Note that if
$\tS$ is a nonsingular matrix submonoid of $\Mat_n(F)$, then
$\mcong{\tS}{P}$ also is a nonsingular submonoid of $\Mat_n(F)$,
for any permutation matrix $P$. On the other hand, these often do
not mix well, as seen in Lemma~\ref{sns} below.

The following example also shows that  nonsingularity need not be preserved under
multiplication  in~$\SLnI$, even when we
conjugate by  diagonal matrices.

\begin{example} (logarithmic notation) \pSkip If $B =  \left(\begin{array}{cc}
 0 & {-\infty}   \\
 {1} & 0
\end{array}\right)$ and  $D =  \left(\begin{array}{cc}
 1 & -\infty   \\
 {-\infty} & -1
\end{array}\right)$, then $BDB^{\operatorname{t}} =  \left(\begin{array}{cc}
 1& 2  \\
2 & 3
\end{array}\right)$ is singular. In view of
Proposition~\ref{cor:A.nb.1}, $B(DB^{\operatorname{t}}D^{-1})$ is
singular.

\end{example}

Here is one consolation.

\begin{lem}\label{monoid1} If $\qI_A \in \tJn,$
then  $\{ A^\nb J  A \ : \  J \in \tJn \} $ is a monoid.
\end{lem}
\begin{proof} By Theorem~\ref{thm:Jn-nb-closed}.
\end{proof}

The situation improves significantly when we restrict
our attention to the submonoid $\SLmonlA$ and the space on which
it acts.
%
 We define $$V_A = \{v \in F^{(n)}: \qIl_A v = v \}.$$

\begin{lem} $\SLmonlA F^{(n)} = V_A = \SLmonlA V_A.$\end{lem}
\begin{proof} If $B \in \SLmonlA$ and $v \in F^{(n)}$, then $\qIl_A
(Bv) = (\qIl_A B) v = Bv.$ On the other hand, if $v \in V_A$ then $v
= \qIl_A v \in \SLmonlA V_A.$ Thus, equality holds since  $\ \SLmonlA F^{(n)}
\subseteq V_A \subseteq \SLmonlA V_A \subseteq \SLmonlA F^{(n)}.$
\end{proof}

\begin{prop} For any nonsingular $A,$ left multiplication by $A^\nb$ yields  a module map
from $V_A$ to $V_{A^\nb}$, which commutes with conjugation by $A$.
\end{prop}
\begin{proof} If $B \in \SLmonlA$, then letting $v' = A^\nb v$ we have~$(A^\nb B  A)v' = (A^\nb B  A)A^\nb v = A^\nb B v.$\end{proof}

%
%
%

\section{Tropical elementary matrices}\label{5}

Unlike the situation over a field, the tropical concepts of
singularity, invertibility, and factorability into elementary
matrices do not coincide, cf. \cite{Row}. Over a field, the fact
that a nonsingular matrix can be written as the product of
elementary matrices means that one can pass between any two
nonsingular matrices using elementary operations. In the tropical
case, even though factorability fails, we show in
Theorem~\ref{nonsingpr1} below that one still can pass between
nonsingular matrices, in a certain sense.

In analogy with the classical definition, we define three types of
tropical elementary matrices of $\SL_n$:

\begin{itemize}

\item[--] \textbf{Transposition matrices}, which switch two rows
(resp.~columns);

\item[--] \textbf{Diagonal multipliers}, which multiply a row
(resp.~a column) by some element of $\tT$;

\item[--] \textbf{Gaussian matrices}, which add one row
(resp.~column), multiplied by a scalar, to
another row (resp.~column).

\end{itemize}

\begin{df} A nonsingular matrix is defined to be \textbf{(tropically) factorizable} if it can be written as a product of tropical elementary matrices.\end{df}

As noted earlier, the product of nonsingular matrices could be
singular. Since the transposition and diagonal multipliers are invertible, the difficulty
must lie in the Gaussian matrices, which are identified with the  $ E_{i,j}(a)$ defined earlier. In the next lemma we pinpoint the elementary
operation that causes a nonsingular  matrix which is non-invertible  to
become singular.

\begin{lem}\label{sns} For every non-invertible matrix $A$ in $\QSLn$,
 there exists an elementary Gaussian matrix $E$  such that~$EA$ is singular.
\end{lem}

\begin{proof}
First we recall that if  $A$ is a factorizable matrix, then we can
find a factorization in which the Gaussian matrices are at the right
of its factorization (see~\cite{Niv2}). Therefore, in view of
Lemma~\ref{factorout} it suffices to prove the lemma for a definite
matrix $A$. Hence, $\per(A)=\one$ is attained solely by the
diagonal.

Since  $A$ is non-invertible,   there exists at least one
off-diagonal entry  $a_{i,j}\ne \zero$. We let~ $E= E_{j,i}(a_{i,j}^{-1})$. Then
\begin{align*}\per(EA)&=\sum_{\sigma\in S_n}a_{1,\sigma(1)}\cdots
a_{i,\sigma(i)}\cdots a_{n,\sigma(n)}+\sum_{\sigma\in
S_n}a_{1,\sigma(1)}\cdots (a_{i,j}^{-1}) a_{i,\sigma(j)}\cdots
a_{n,\sigma(n)}\\ & =\per(A)+\sum_{\sigma\in
S_n}a_{1,\sigma(1)}\cdots (a_{i,j}^{-1}) a_{i,\sigma(j)}\cdots
a_{n,\sigma(n)}.\end{align*}   The summand in the right side given by  $\sigma = (i,j)$ is
$\one$, which together with $\per(A)$ yields $\one^\nu$. Moreover,
by \cite[Theorem
3.5]{IzhakianRowen2008Matrices}, any larger dominant term on the right
sum must be ghost.   Since $\per(A)=\one$, the assertion follows.
\end{proof}

We recall the well-known connection between tropical matrices and
digraphs. Any $n \times n$ matrix $A$ is associated with  a weighted
digraph $G_A$ over $n$ vertices having edge $(i,j)$ of weight
$a_{i,j}$ whenever  $a_{i,j} \neq \zero$
cf.~\cite[\S3.2]{IzhakianRowen2008Matrices}. From this   viewpoint
the $(i,j)$-entry of the matrix $\adj{A}$ equals the maximal weight
of all paths from $i$ to~$j$ in the graph $G_{\adj{A}}$. We utilize
this identification and work with nonsingular definite matrices, in
which case $A^\nb \cong _\nu A^\nbnb$ by \cite[Corollary~6.2]{Niv}.
 A path is called \textbf{simple} if each vertex
appears only once.

\begin{prop}\label{ED}
For any  matrix $A \in \QSLn$ there exists a product $E$ of
elementary Gaussian matrices  such that $A^{\nbnb}=E A.$
\end{prop}

\begin{proof} In view of Lemma~\ref{factorout} we may assume that~$A$ is definite; indeed, writing~$A = PA_1$, for $A_1$ definite, we
would have~$(P A_1)^\nbnb = (A_1^\nb P^{-1})^\nb = PA_1 ^\nbnb =
PEA_1 = (PEP^{-1})PA_1.$

 Now let~$A=(A_{i,j}) , A^{\nb}=(A^{\nb}_{i,j})$ and let~$A^{\nbnb}=(A^{\nbnb}_{i,j})$.
 Then~$A^{\nbnb} \lmodg A$  by Remark~\ref{somefac}(v).
Since~$A$ is nonsingular definite,  we have~$\per(A) = \one =
\prod _i A_{i,i}$ where its dominant permutation is the identity. It follows
that any nontrivial cycle has weight~$< \one$.  Any product
including such a cycle is strictly dominated by the product where
this cycle is replaced by the (weight $\one$) identity permutation
on the corresponding set of indices.

 Thus, we may assume that each~$(i,j)$-entry of~$\adj{A}$
is~$\nu$-equivalent to the sum of weights of simple paths from~$i$
to~$j$:
$$A^{\nb}_{i,j}\nucong \sum_{\substack{\pi\in S_n:\\
\pi(j)=i}}A_{i,\pi(i)}A_{\pi(i),\pi^2(i)}\cdots
A_{\pi^{-1}(j),j}=A_{i,j}+\sum_{\substack{(i\ j)\ne\pi\in S_n:\\
\pi(j)=i}}A_{i,\pi(i)}A_{\pi(i),\pi^2(i)}\cdots
A_{\pi^{-1}(j),j}.$$
Applying  this identity to~\cite[Corollary 6.2]{Niv2}:~$A^{\nb}\cong_\nu A^{\nbnb}$,
yields for every entry such that~$A^{\nbnb}_{i,j}\ne A_{i,j}$:

 $$\bigg[E_{i,j}(A^{\nbnb}_{i,j})A\bigg]_{k,\ell}=\begin{cases}
A_{k,\ell}\ ,\ \forall k\ne i\\
A_{i,\ell}+A^{\nbnb}_{i,j}A_{j,\ell}\ ,\  k= i
\end{cases}.$$
When $j=\ell$, we get $A_{i,\ell}+A^{\nbnb}_{i,j}\cdot\one=A^{\nbnb}_{i,j}$. If $j\ne\ell$, then
$$A^{\nbnb}_{i,j}A_{j,\ell}=\sum_{\substack{(i\ j)\ne\pi\in S_n:\\
\pi(j)=i}}A_{i,\pi(i)}A_{\pi(i),\pi^2(i)}\cdots
A_{\pi^{-1}(j),j}\ \cdot \ A_{j,\ell}$$
is either a closed path dominated by $A_{i,i}=\one$ when $\ell=i$, or is the product of cycles with an elementary path from $i$ to $\ell$
 dominated by $A^{\nbnb}_{i,\ell}$.
Therefore, applying $E_{i,j}(A^{\nbnb}_{i,j})$ for every $j\ne i$ such that $A^{\nbnb}_{i,j}\ne A_{i,j} $ transforms all entries of $A$ to entries of $A^{\nbnb}$
and we get
$$A^{\nbnb}=\bigg(\prod_{\substack{j< i:\\ A^{\nbnb}_{i,j}\ne A_{i,j}}} E_{i,j}(A^{\nbnb}_{i,j})\bigg)\bigg(\prod_{\substack{j> i:\\ A^{\nbnb}_{i,j}\ne A_{i,j}}} E_{i,j}(A^{\nbnb}_{i,j})\bigg)A,$$
that is, where the elementary operations~$E_{i,j}(A^{\nbnb}_{i,j})$
 are applied to upper entries first, lower entries later, and $\{i,j\}$ is lexicographically  ordered.

\end{proof}

Let $A$ and $B$ be nonsingular matrices. Over a field, in classical
linear algebra, $A$ and $B$ can be written as products of elementary
matrices.
Thus,  one can pass from $A$ to $B$ by applying elementary
operations. In the tropical case, whereas we do not have
factorizability into elementary matrices,
cf.~\cite[Example~4.5]{Niv2}, we do have the second implication,
described in the following theorem.

\begin{thm}\label{nonsingpr1} For any two nonsingular matrices $A,B$, there
exist matrices $E_1,E_2,E_3,E_4$ which are products of
  elementary matrices  of $\QSLn$, such that $E_1AE_2=E_3BE_4$.
\end{thm}

\begin{proof} Using Lemma~\ref{factorout}, we write~$A^{\nbnb}=\brA^{\nbnb} P$
and~$B^{\nbnb}=\brB^{\nbnb}  Q $ where~$\brA,\brB$ are  definite, and~$P,Q$ are invertible matrices chosen so that~$A = \bar{A}P$ and~$B=
\bar{B}Q.$

By Remark~\ref{somefac}(v),
the matrices~${\brA}^{\nbnb}$ and ${\brB}^{\nbnb}$
respectively ghost-surpass~$\brA$ and~$\brB$, and
noting that~$P^{\nbnb} = P$ and~$Q^{\nbnb} = Q,$ %
  we get that $$A^{\nbnb}=\brA^{\nbnb}  P  \lmodg A=\brA P \
\text{ and }\ B^{\nbnb}=\brB^{\nbnb}  Q   \lmodg B=\brB Q.$$
Recalling~\cite[Lemma~6.5]{Niv2}, the matrices~${\brA}^{\nbnb}$ and ${\brB}^{\nbnb}$ are factorizable,
and therefore, $A^{\nbnb}$ and $B^{\nbnb}$ are factorizable.
Clearly~$IA^{\nbnb} B^{\nbnb}=A^{\nbnb} B^{\nbnb} I,$ which
provides the assertion for~$A^{\nbnb}$ and~$ B^{\nbnb}.$

 We denote by~$E,E'$ the elementary products
such that~${\brA}^{\nbnb}=E\brA\ \text{ and }\
{\brB}^{\nbnb}=E'\brB,$ whose existence are  guaranteed by
Proposition~\ref{ED}, and have
$$E  A   B^{\nbnb}=E\brA P B^{\nbnb}={\brA}^{\nbnb}P{B}^{\nbnb}=A^{\nbnb}B^{\nbnb}=A^{\nbnb} {\brB}^{\nbnb}Q=A^{\nbnb} E' \brB Q=A^{\nbnb}E'  B.$$
But~$ E,\ B^{\nbnb},$~$A^{\nbnb},$ and~$ E'$ are products of elementary matrices.
\end{proof}

%
%
%
%


\begin{thebibliography}{10} 



\bibitem{ABG}
M.~Akian, R.~Bapat, and  S.~Gaubert.
\newblock Max-plus algebra,
\newblock In:  L. Hogben, R. Brualdi, A. Greenbaum, R. Mathias  (eds.)
{\em Handbook of Linear Algebra}. Chapman and Hall, London, 2006.




\bibitem{AGG}
M.~Akian,   S.~Gaubert, and A.~Guterman.
\newblock Linear independence over tropical semirings and beyond.
\newblock In {\em Tropical and  Idempotent Mathematics}, G.L. Litvinov and S.N.
Sergeev, (eds.),
\newblock {\em Contemp. Math.},  495:1--38, 2009.

  \bibitem{AGG1}
M.~Akian,   S.~Gaubert, and A.~Guterman. Tropical Cramer
determinants
    revisited. In {\em Tropical and idempotent mathematics and
    applications},
   Contemp. Math., Amer. Math. Soc., Providence, RI, 616:1--45,
   2014.



\bibitem{BCOQ}     F.~Baccelli, G.~Cohen, G.J.~Olsder, and J.P.~Quadrat. Synchronization and linearity. Wiley
    Series in Probability and Mathematical Statistics: Probability and Mathematical Statistics.
    John Wiley and Sons, Ltd., Chichester, 1992.

\bibitem{B} P. Butkovi\v{c}. Max-algebra: the linear algebra of combinatorics?,
{\em Linear Algebra Appl.}, 367: 313--335, 2003.

\bibitem{Cu} R.A.~Cuninghame-Green. {\em Minimax algebra},
Lecture Notes in Economics and Mathematical Systems 166,
Springer-Verlag, Berlin, 1979.

\bibitem{DoO} D. Dol\u{z}na and P. Oblak.
\newblock  Invertible and nilpotent matrices over antirings. {\em Lin. Algebra and Appl.}, 430(1):271--278, 2009.

%

\bibitem{G.Ths}
  S.~Gaubert. {\em Th\'{e}orie des syst\`{e}mes lin\'{e}aires
dans les dio\"{i}des}. PhD dissertation, School of Mines. Paris,
July  1992.

 \bibitem{TTNN}
S.~Gaubert, and A.~Niv. Tropical totally positive matrices.
Preprint at arXiv:1606.00238, 2016.

  \bibitem{GoM}     M.~Gondran and M.~Minoux. Graphs, dioids and semirings, volume 41 of Operations
    Research/Computer Science Interfaces Series. Springer, New York, 2008.

\bibitem{GIMM} P. Guillon, Z. Izhakian, J. Mairesse, and G. Merlet.
The ultimate rank of tropical matrices,    {\em J. of Algebra}, 437:222--248, 2015.

\bibitem{Gun} J.~Gunawdera. An introduction to indempotency, Basic
Reseach Institute in the Mathematical Sciences HP Laboratories,
Bristol HPL-BRIMS, 96-24, 1996.



\bibitem{zur05TropicalAlgebra}
Z.~Izhakian.
\newblock Tropical arithmetic and matrix algebra.
\newblock {\em Comm. in Algebra}  {37}(4):1445--1468, {2009}.

\bibitem{IdMax}
Z.~Izhakian.
\newblock Semigroup identities of  tropical matrix semigroups of maximal rank, {\em Semigroup Forum}, 92(3):712--732, 2016.



\bibitem{IJK} Z.~Izhakain, M.~Johnson, and  M. Kambites.
Tropical matrix groups, {\em Semigroup Forum}, to appear, 2016.

\bibitem{IJK2} Z.~Izhakain, M.~Johnson, and  M. Kambites.
Pure dimension and projectivity of tropical convex sets, {\em Adv. in Math.},  303:1236--1263, 2016.

\bibitem{IKR-LinAlg}
Z.~Izhakian, M.~Knebusch, and L.~Rowen.
\newblock Supertropical linear algebra.
\newblock {\em Pacific J. of Math.},  266(1):43--75, 2013.

%

\bibitem{IzhakianRowen2009TropicalRank}
Z.~Izhakian and L.~Rowen.
\newblock The tropical rank of a tropical matrix.
\newblock {\em Comm. in Algebra},  {37}(11):{3912--3927}, 2009.

\bibitem{IzhakianRowen2007SuperTropical}
Z.~Izhakian and L.~Rowen.
\newblock {Supertropical algebra}.
\newblock   {\em Adv. in Math.}, 225(4):2222--2286, 2010.

\bibitem{IzhakianRowen2008Matrices}
Z.~Izhakian and L.~Rowen.   \newblock  {Supertropical matrix
algebra.} \newblock  {\em Israel J. of Math.},
  182(1):383--424, 2011.

\bibitem{IzhakianRowen2009Equations}
Z.~Izhakian and L.~Rowen.
 Supertropical matrix algebra II: Solving tropical
equations, \emph{Israel J. of Math.}, 186(1): 69--97, 2011.

%

\bibitem{Ku}
J.~Kuntzman.
\newblock Th\'{e}orie des R\'{e}seau, Dunot, Paris, 1972.
\bibitem{merl10}
G.~Merlet.
\newblock Semigroup of matrices acting on the max-plus projective space,
\newblock {\em Linear Algebra Appl.},  432(8):1923--1935, 2010.


\bibitem{Niv} A.~Niv. \newblock {On pseudo-inverses of matrices and their
characteristic polynomials in supertropical algebra}, \newblock
\emph{Linear Algebra Appl.}  471: 264--290, 2015.



\bibitem{Niv2}
A. Niv. \newblock {Factorization of tropical matrices}, \newblock
\emph{J.~Algebra Appl.}, 13(1):1350066:1--26, 2014.

\bibitem{Pl} M. Plus. Linear systems in (max; +)-algebra. In Proceedings of the 29th Conference on
    Decision and Control, Honolulu, Dec. 1990.


\bibitem{RS} C.~Reutenauer and H.~Straubing. {Inversion of matrices over a
commutative semiring}, {\em J.~Algebra}, {88}:350--360,  1984.

\bibitem{Row}
L.H.~Rowen. {\em Graduate algebra: A noncommutative view}
\newblock {\em Semigroups and Combinatorial Applications}.
\newblock American Mathematical Society, 2008.

\bibitem{Rut} D.E.~Rutherford. Inverses of Boolean matrices,
    Proceedings of the Glasgow Mathematical Association, 6:49--53,  1963.


\bibitem{Str} G.~Strang. {\em Linear algebra and its applications},
\newblock Academic Press, 1976.

    \bibitem{St} H.~Straubing. A combinatorial proof of the Cayley-Hamilton
theorem, {\em Discrete Math.},  43(2-3):273--279, 1983.

\bibitem{Tu} A.M.~Turing. Rounding-Off Errors in Matrix Processes, {\em The Quarterly J. of Mechanics and Applied Math.}, 1:287--308, 1948.
 \end{thebibliography}
\end{document}